\documentclass[a4paper,reqno]{article}

\usepackage[utf8]{inputenc} 
\usepackage[T1]{fontenc} 
\usepackage{lmodern} 
\usepackage[american]{babel}
\usepackage{csquotes}

\usepackage{verbatim} 
\usepackage{authblk} 
\usepackage{calc}
\usepackage{amsmath}
\usepackage{amsfonts}
\usepackage{graphicx}
\usepackage{float}
\usepackage{xcolor}
\usepackage{enumitem}
\setlist[enumerate]{leftmargin=.5in} 
\setlist[itemize]{leftmargin=.5in} 
\providecommand{\keywords}[1]{\textbf{Keywords: } #1} 

\usepackage[
pdfmenubar=true,
pdftitle={Computing covariant Lyapunov vectors in Hilbert spaces},
pdfauthor={Florian Noethen},
colorlinks=false,
linkcolor=black,
citecolor=blue,
filecolor=cyan,
urlcolor=magenta,
hidelinks=true,
raiselinks=false,
hyperfootnotes=true
]{hyperref}
\usepackage[capitalize,nameinlink]{cleveref} 

\def\natural{\mathbb{N}}
\def\integer{\mathbb{Z}}
\def\real{\mathbb{R}}

\let\originalleft\left
\renewcommand{\left}{\mathopen{}\mathclose\bgroup\originalleft} 
\let\originalright\right\renewcommand{\right}{\aftergroup\egroup\originalright} 

\usepackage{amsthm}
\usepackage{amssymb}
\newtheorem*{theorem*}{Theorem} 
\newtheorem{theorem}{Theorem}[section] 
\newtheorem{lemma}[theorem]{Lemma} 
\newtheorem{proposition}[theorem]{Proposition}
\newtheorem{definition}[theorem]{Definition}

\newtheorem{remark}[theorem]{Remark}
\newtheorem{corollary}[theorem]{Corollary}
\newtheorem*{conjecture*}{Conjecture} 

\Crefname{conjecture}{Conjecture}{Conjectures}

\usepackage{tikz}
\usetikzlibrary{arrows}
\usepackage{environ} 
\makeatletter
\newsavebox{\measure@tikzpicture}
\NewEnviron{scaletikzpicturetowidth}[1]{%
	\def\tikz@width{#1}%
	\def\tikzscale{1}\begin{lrbox}{\measure@tikzpicture}%
		\BODY
	\end{lrbox}%
	\pgfmathparse{#1/\wd\measure@tikzpicture}%
	\edef\tikzscale{\pgfmathresult}%
	\BODY
}
\makeatother

\usepackage[
	sorting = nyt,
	sortcites,
	giveninits = true,
	isbn = false,
	doi = false,
	clearlang = true,
	backend = biber,
	citestyle = numeric-comp  
]{biblatex}
\DeclareNameAlias{author}{last-first}
\AtEveryBibitem{\clearfield{pagetotal}} 
\AtEveryBibitem{\clearfield{note}} 
\DeclareSourcemap{
	\maps[datatype=bibtex]{
		\map{
			\pernottype{misc} 
			\step[fieldset=url, null] 
			\step[fieldset=urldate, null] 
		}
	}
}
\DeclareRedundantLanguages{ 
	English,
	english,
	german,
	french,
	eng
}{
	english,
	german,
	ngerman,
	french
}
\renewbibmacro{in:}{ 
	\ifentrytype{article}{}{\printtext{\bibstring{in}\intitlepunct}}%
}

\begin{document}
	
	\title{Computing covariant Lyapunov vectors in Hilbert spaces}
\author{Florian Noethen
	\thanks{Fachbereich Mathematik, Universität Hamburg, Bundesstraße 55, 20146 Hamburg, Germany (\href{mailto:florian.noethen@uni-hamburg.de}{florian.noethen@uni-hamburg.de}).}
}
\date{\today} 
\maketitle

\begin{abstract}
	
	Covariant Lyapunov vectors (CLVs) are intrinsic modes that describe long-term linear perturbations of solutions of dynamical systems. With recent advances in the context of semi-invertible multiplicative ergodic theorems, existence of CLVs has been proved for various infinite-dimensional scenarios. Possible applications include the derivation of coherent structures via transfer operators or the stability analysis of linear perturbations in models of increasingly higher resolutions.\par 
	
	We generalize the concept of Ginelli's algorithm to compute CLVs in Hilbert spaces. Our main result is a convergence theorem in the setting of [González-Tokman, C. and Quas, A., \textit{A semi-invertible operator Oseledets theorem}, Ergodic Theory and Dynamical Systems, 34.4 (2014), pp. 1230-1272]. The theorem relates the speed of convergence to the spectral gap between Lyapunov exponents. While the theorem is restricted to the above setting, our proof requires only basic properties that are given in many other versions of the multiplicative ergodic theorem.
		
\end{abstract}

\keywords{covariant Lyapunov vectors (CLVs); multiplicative ergodic theorem, Ginelli algorithm; Hilbert spaces}

\hspace{1em}


\tableofcontents

\hypersetup{linkcolor=red}
	
	\section{Introduction}\label{sec:Intro}

	\emph{Covariant Lyapunov vectors} (CLVs) characterize the asymptotically most expanding directions in tangent space along trajectories in dynamical systems. They have been described as the ``physically relevant'' modes in dissipative systems \cite{TakeuchiYangGinelliRadonsChate.2011} and have been used to detect coherent structures, i.e., slow mixing sets, via the Perron-Frobenius operator \cite{FroylandLloydQuas.2013,FroylandLloydQuas.2010,GonzalezTokman.2018}, the dual of the Koopman operator. Recent research on coherent structures includes the analysis of large scale features of the ocean and atmosphere relevant for climate \cite[chapter 6]{GonzalezTokman.2018}. Apart from techniques involving transfer operators, CLVs have been used directly to analyze instabilities in coupled models. Two examples are the assessment of long-term predictability in an ocean-atmosphere model \cite{VannitsemLucarini.2016} and the decoupling of instabilities into modes associated to different time-scales to analyze mixing in a two-scale Lorenz 96 model \cite{CarluGinelliLucariniPoliti.2019}.\par 
	
	In this article, we generalize the concept of \textit{Ginelli's algorithm} \cite{GinelliPoggiTurchiChateLiviPoliti.2007} to compute CLVs for even infinite-dimensional settings. Our main contribution is a convergence result in the context of Hilbert spaces.\par
	
	To prove convergence and to guarantee existence of CLVs, we need the \textit{multiplicative ergodic theorem} (MET). While the original MET from 1968 is due to Oseledets \cite{Oseledets.1968}, until today various other versions emerged (e.g., see \cite{Arnold.1998,BarreiraSilva.2005,Blumenthal.2015,Doan.2010,FroylandLloydQuas.2010,FroylandLloydQuas.2013,GonzalezTokmanQuas.2014,GonzalezTokmanQuas.2015,LianLu.2010,Ruelle.1982}). Their application ranges from deterministic to stochastic systems in finite and infinite dimensions. Moreover, there is a distinction between non-invertible and invertible versions. Noninvertible versions only derive an \emph{Oseledets filtration}, whereas invertible versions yield an \textit{Oseledets splitting}. The corresponding spaces of the splitting are called \textit{Oseledets spaces} and give rise to the CLVs. Aside from non-invertible and invertible versions, there is a third, more recent class of \emph{semi-invertible} METs. Versions of this class still provide an Oseledets splitting and, for instance, can be applied to transfer operators. Several semi-invertible METs were proved for infinite-dimensional scenarios \cite{GonzalezTokmanQuas.2014,GonzalezTokmanQuas.2015,FroylandLloydQuas.2013}. Here, we follow the semi-invertible MET from \cite{GonzalezTokmanQuas.2014}. For an overview of the history of METs and applications to transfer operators in the context of non-autonomous systems, we highly recommend reading \cite{GonzalezTokman.2018}.\par 
	
	To prove existence of an Oseledets splitting, \cite{GonzalezTokmanQuas.2014} pushes forward a set of special complements of the Oseledets filtration from the far past to the present state along trajectories. Indeed, it turns out that complements of the Oseledets filtration will align with sums of the first Oseledets spaces in forward-time generically. This idea was used by Ginelli et al. to compute CLVs \cite{GinelliPoggiTurchiChateLiviPoliti.2007, GinelliChateLiviPoliti.2013} or, more generally, Oseledets spaces. The first part of their algorithm approximates sums of Oseledets spaces through past data, while the second part uses future data to extract an approximation of CLVs from the former approximations. Its dynamical description distinguishes Ginelli's algorithm from other approaches \cite{WolfeSamelson.2007,KuptsovParlitz.2012,FroylandHuelsMorrissWatson.2013} and makes it applicable to a wide range of scenarios.\par 
	
	We generalize an existing convergence result for invertible systems in finite dimensions \cite{Noethen.2019} to a broader setting on Hilbert spaces. While \cite{Noethen.2019} heavily focuses on a singular value decomposition of the linear propagator, as it appears in the proof of the MET from \cite{Arnold.1998}, we use a purely dynamical approach. With the help of \emph{well-separating common complements} \cite{Noethen.2019b} of the Oseledets filtration, we are able to prove the following: 	
	
	\begin{theorem*}
		In the setting of \cite{GonzalezTokmanQuas.2014} for Hilbert spaces, Ginelli's algorithm convergence for almost every initial configuration. The convergence is exponentially fast with a rate given by the spectral gap between corresponding Lyapunov exponents.
	\end{theorem*}

	Even though the theorem is proved for Hilbert spaces, many arguments of the proof hold for Banach spaces. In fact, we formulate the majority of the theory and tools for Banach spaces. Moreover, we only require basic properties stated in the MET. Hence, our results may be translated to other scenarios apart from \cite{GonzalezTokmanQuas.2014}.\par 
	
	We begin the article by laying foundations. \cref{subsec:Grassmannians} introduces \emph{Grassmannians}. They naturally appear in the MET and are essential for our convergence proof later on. In \cref{subsec:MET} we state the MET from \cite{GonzalezTokmanQuas.2014}. We extract basic asymptotic properties found in the proof of the MET. Those properties are not unique to \cite{GonzalezTokmanQuas.2014}, but can also be found in other versions of the MET.\par 
	
	\cref{sec:ComputingCLVs} presents our new research. After defining Ginelli's algorithm in \cref{subsec:Ginelli}, we devote the remaining subsections to prove our convergence theorem. \cref{subsec:ForwardTime} treats forward propagation, whereas \cref{subsec:BackwardTime} adds backward propagation along certain subspaces to the forward propagation. Both subsections are formulated in the context of maps on Banach spaces. Hence, they can be applied to a potentially larger class of systems than given by \cite{GonzalezTokmanQuas.2014}. In \cref{subsec:ConvergenceProof} we combine the derived tools to come up with a convergence proof of Ginelli's algorithm on Hilbert spaces.

\bigskip
	
	\section{Setting}\label{sec:Setting}

	Before introducing Ginelli's algorithm, we need to derive Oseledets spaces and their asymptotic properties from the MET. Oseledets spaces are finite-dimensional subspaces complemented by closed subspaces of the Oseledets filtration. In particular, they are elements of the Grassmannian of $X$. Understanding pairs of complementary subspaces from the MET is fundamental for our convergence proof, since they encode the different asymptotic growth rates of linear perturbations.\par
	
	Let us start by introducing Grassmannians.

\subsection{Grassmannians}\label{subsec:Grassmannians}
	
	\begin{definition}\label[definition]{def:Grassmannian}
		
		Let $X$ be a Banach space. The \emph{Grassmannian} $\mathcal{G}(X)$ is the set of closed complemented subspaces of $X$, i.e., closed subspaces $V\subset X$ such that there is a closed subspace $W\subset X$ with $X=V\oplus W$. It contains $\mathcal{G}_k(X)$, the set of $k$-dimensional subspaces, and $\mathcal{G}^k(X)$, the set of closed subspaces of codimension $k$.
		
	\end{definition}
	
	The Grassmannian $\mathcal{G}(X)$ can be equipped with a metric $\textnormal{d}_{\mathcal{G}}(V,W)$ via the Hausdorff distance between $V\cap B$ and $W\cap B$, where $B$ denotes the closed unit ball in $X$ \cite[appendix B]{GonzalezTokmanQuas.2014}:
	\begin{align*}
		\textnormal{d}_{\mathcal{G}}(V,W)&:=\textnormal{d}_H(V\cap B,W\cap B)\\
		&=\max\left(\underset{v\in V\cap B}{\sup}\,\textnormal{d}(v,W\cap B),\underset{w\in W\cap B}{\sup}\,\textnormal{d}(w,V\cap B)\right)\\
		&=\max\left(\underset{v\in V\cap B}{\sup}\,\underset{w\in W\cap B}{\inf}\,\|v-w\|,\underset{w\in W\cap B}{\sup}\,\underset{v\in V\cap B}{\inf}\,\|w-v\|\right)
	\end{align*}
	for $V,W\in\mathcal{G}(X)$. Another possible metric is given by exchanging $B$ with the unit sphere $S$ in the above definition \cite[chapter IV, §2.1]{Kato.1995}. As our definition of $\textnormal{d}_{\mathcal{G}}$ lies between the ``gap'' and the metric from \cite{Kato.1995}, both metrics induce the same topology and make $\mathcal{G}(X)$ into a complete metric space.\par
	
	The symmetry of $\textnormal{d}_{\mathcal{G}}$ is an immediate consequence of the symmetric definition of $\textnormal{d}_{H}$. This kind of definition is necessary, since $\sup_{v\in V\cap B}\,\textnormal{d}(v,W\cap B)$ and $\sup_{w\in W\cap B}\,\textnormal{d}(w,V\cap B)$ are different in general. However, if one term is small, then so is the other \cite[lemma B.7]{GonzalezTokmanQuas.2014}.
	
	\begin{lemma}\label[lemma]{lemma:SymmetryOfCloseness}
		
		If $V,W\in\mathcal{G}_k(X)$ are subspaces of dimension $k$, then
		\begin{equation*}
			\sup_{v\in V\cap B}\,\textnormal{d}(v,W\cap B)=:r<3^{-k}/4\hspace{1em}\implies\hspace{1em}\textnormal{d}_{\mathcal{G}}(V,W)<4\cdot 3^k r.
		\end{equation*}
		If $V,W\in\mathcal{G}^k(X)$ are closed subspaces of codimension $k$, then
		\begin{equation*}
		\sup_{v\in V\cap B}\,\textnormal{d}(v,W\cap B)=:r<3^{-k}/8\hspace{1em}\implies\hspace{1em}\textnormal{d}_{\mathcal{G}}(V,W)<8\cdot 3^k r.
		\end{equation*}
		
	\end{lemma} 

	Thus, when investigating convergence inside $\mathcal{G}_k(X)$ or $\mathcal{G}^k(X)$, it is enough to estimate only one of the two terms in the definition of $\textnormal{d}_{H}$.\par

	Ultimately, we want to approximate Oseledets spaces, which are special finite-dimensional complements of spaces of the Oseledets filtration. Hence, we will be working with tuples of the set
	\begin{equation*}
		\textnormal{Comp}_k(X):=\{(Y,Z)\in\mathcal{G}_k(X)\times\mathcal{G}^k(X)\ |\ X=Y\oplus Z\}
	\end{equation*}
	for $k\in\natural$. Given such a tuple, each $x\in X$ can be written uniquely as $x=y+z$ according to the associated splitting. In particular, we get two projections $\Pi_{Y||Z}:X\to Y$ and $\Pi_{Z||Y}:X\to Z$, which are bounded linear operators by the closed graph theorem. It can be shown that they are stable with respect to perturbations of the tuple $(Y,Z)$ \cite[lemma B.18]{GonzalezTokmanQuas.2014}.
	
	\begin{lemma}\label[lemma]{lemma:ProjectionStable}
		
		The mapping $\textnormal{Comp}_k(X)\to\mathcal{L}(X)$ given by $(Y,Z)\mapsto\Pi_{Y||Z}$ is continuous, where $\textnormal{Comp}_k(X)$ has the product topology induced by $\mathcal{G}(X)$ and the space $\mathcal{L}(X)$ of bounded linear operators on $X$ is equipped with the norm topology.
		
	\end{lemma}
 
	Finally, we need one more concept for Grassmannians. Given a sequence of subspaces $(V_n)_{n\in\natural}\subset\mathcal{G}^k(X)$, we ask for common complements, i.e., subspaces $W\subset X$ with $(W,V_n)\in\textnormal{Comp}_k(X)$ for all $n$. Natural questions are the existence and quantity of common complements. A recent paper links these questions to quality assumptions \cite{Noethen.2019b}. A complement is called \emph{well-separating} if the degree of transversality\footnote{Our notion of the degree of transversality coincides with the sine of the minimal angle between two subspaces of a Banach space from \cite{Blumenthal.2015}.} $\inf_{x\in W,\, \|x\|=1}\,\textnormal{d}(x,V_n)$ of the tuple $(W,V_n)$ decays at most subexponentially with $n$. Well-separating common complements can be used without interfering on exponential scales that are important for our convergence proof later on. 
	
	\begin{definition}\label[definition]{def:WellSeparating}
		
		Let $(V_n)_{n\in\natural}\subset\mathcal{G}^k(X)$ be given. A common complement $W\in\mathcal{G}_k(X)$ of $(V_n)_{n\in\natural}$ is called \emph{well-separating} w.r.t. $(V_n)_{n\in\natural}$ if
		\begin{equation}
			\lim_{n\to\infty}\frac{1}{n}\log \underset{x\in W\cap S}{\inf}\,\textnormal{d}(x,V_n)=0.
		\end{equation}
		
	\end{definition}

	Using the concept of \emph{prevalence} \cite{OttYorke.2005}, \cite{Noethen.2019b} proves that almost every tuple of vectors induces a well-separating common complement if there exists at least one such complement. Since existence is guaranteed in Hilbert spaces, we have the following theorem.
	
	\begin{theorem}\label[theorem]{thm:WellSeparating}
		
		Let $H$ be a Hilbert space and let $(V_n)_{n\in\natural}\subset\mathcal{G}^k(H)$. Almost every tuple $(x_1,\dots,x_k)\in H^k$ induces a well-separating common complement of $(V_n)_{n\in\natural}$ via $\textnormal{span}(x_1,\dots,x_k)$.
		
	\end{theorem}

	This theorem plays a crucial role in our convergence proof. In fact, existence of well-separating common complements in Banach spaces would suffice to prove a version of \cref{thm:WellSeparating} for Banach spaces. Hence, we formulate our results in \cref{sec:ComputingCLVs} at the level of Banach spaces while leaving open the question of generality until \cref{subsec:ConvergenceProof}, where we restrict ourselves to Hilbert spaces.

\subsection{Multiplicative ergodic theorem}\label{subsec:MET}

	METs describe asymptotic behavior of linear perturbations of trajectories in dynamical systems in terms of an Oseledets filtration or in terms of an Oseledets splitting. We state the semi-invertible MET from \cite{GonzalezTokmanQuas.2014} to extract basic asymptotic properties that are needed to compute Oseledets spaces. Prior to that, let us recall a few preliminary facts from \cite[section 2.1]{GonzalezTokmanQuas.2014}.\par 
	
	Our choice of MET requires a \textit{strongly measurable random dynamical system} $\mathcal{R}=(\Omega,\mathcal{F},\mathbb{P},\sigma,X,\mathcal{L})$. It consists of a base (flow) and a cocycle describing the tangent linear dynamics. The \textit{base} $\sigma:\Omega\to\Omega$ is a probability-preserving transformation of a Lebesgue space $(\Omega,\mathcal{F},\mathbb{P})$. It is linked to the cocycle via the \textit{generator}, which is a strongly measurable map $\mathcal{L}:\Omega\to L(X)$, i.e., $\mathcal{L}(.)x:\Omega\to X$ is $(\mathcal{F},\mathcal{B}_X)$-measurable for every $x\in X$. Iterative applications of $\mathcal{L}$ along trajectories yield the \textit{cocycle} $\mathcal{L}_{\omega}^{(n)}:=\mathcal{L}(\sigma^{n-1}\omega)\circ\dots\circ\mathcal{L}(\omega)$. Moreover, we call the random dynamical system \textit{separable} if the Banach space $X$ is separable.\par 
	
	Given a bounded linear operator $A\in L(X)$, we define the \emph{index of compactness} of $A$ as
	\begin{multline*}
		\|A\|_{ic(X)}:=\\
		\inf\{r>0\ |\ A(B)\textnormal{ can be covered by finitely many balls of radius }r\}.
	\end{multline*}
	
	\begin{proposition}\label[proposition]{prop:IndexOfCompactness}
		
		Let $\mathcal{R}=(\Omega,\mathcal{F},\mathbb{P},\sigma,X,\mathcal{L})$ be a separable strongly measurable random dynamical system such that $\log^+\|\mathcal{L}(\omega)\|\in L^1(\Omega,\mathcal{F},\mathbb{P})$, where $\log^+(t):=\max(0,\log t)$.\par
		
		For $\mathbb{P}$-a.e. $\omega\in \Omega$, the \emph{maximal Lyapunov exponent}
		\begin{equation}
			\lambda(\omega):=\lim_{n\to\infty}\frac{1}{n}\log\|\mathcal{L}^{(n)}_{\omega}\|
		\end{equation}
		and the \emph{index of compactness}
		\begin{equation}
			\kappa(\omega):=\lim_{n\to\infty}\frac{1}{n}\log\|\mathcal{L}^{(n)}_{\omega}\|_{ic(X)}
		\end{equation}
		exist. Furthermore, $\lambda$ and $\kappa$ are measurable and $\sigma$-invariant.\par
		
		If $\sigma$ is ergodic, then $\lambda$ and $\kappa$ are constant $\mathbb{P}$-almost everywhere. Denote those constants by $\lambda^*$ and $\kappa^*$. It holds $\kappa*\leq\lambda*<\infty$.
		
	\end{proposition}

	We call a separable strongly measurable random dynamical system with ergodic base \textit{quasi compact} if $\kappa^*<\lambda^*$. For such systems, \cite{Doan.2010} derived the existence of an Oseledets filtration as a corollary of the two-sided MET by Lian and Lu \cite{LianLu.2010}. If we additionally assume that the base is invertible, \cite{GonzalezTokmanQuas.2014} proves a semi-invertible MET with a splitting that is similar to the Oseledets splitting obtained in fully invertible METs. 
	
	\begin{theorem}\label[theorem]{thm:MET}
		
		Let $\mathcal{R}=(\Omega,\mathcal{F},\mathbb{P},\sigma,X,\mathcal{L})$ be a separable strongly measurable random dynamical system over an ergodic invertible base such that $\log^+\|\mathcal{L}(\omega)\|\in L^1(\Omega,\mathcal{F},\mathbb{P})$. Furthermore, assume that $\mathcal{R}$ is quasi-compact.\par 
		
		There exist $1\leq l\leq\infty$ \emph{exceptional Lyapunov exponents} $\lambda^*=\lambda_1>\dots>\lambda_l>\kappa^*$ (or if $l=\infty$: $\lambda_1>\lambda_2>\dots>\kappa^*$ and $\lim_{n\to\infty}\lambda_n=\kappa^*$), multiplicities $m_1,\dots,m_l\in\natural$, and a unique, measurable splitting of $X$ into closed subspaces
		\begin{equation*}
			X=\bigoplus_{j=1}^lY_j(\omega)\oplus V(\omega)
		\end{equation*}
		defined on a $\sigma$-invariant subset $\Omega'\subset\Omega$ of full $\mathbb{P}$-measure such that the following hold for $\omega\in\Omega'$:
		\begin{enumerate}
			\item the splitting is equivariant, i.e., $\mathcal{L}(\omega)V(\omega)\subset V(\sigma\omega)$ and $\mathcal{L}(\omega)Y_j(\omega)= Y_j(\sigma\omega)$,
			\item $\dim Y_j(\omega)=m_j$,
			\item $\lim_{n\to\infty}(1/n)\log\|\mathcal{L}_{\omega}^{(n)}y\|=\lambda_j$ for $y\in Y_j(\omega)\setminus\{0\}$,
			\item $\limsup_{n\to\infty}(1/n)\log\|\mathcal{L}_{\omega}^{(n)}v\|\leq\kappa^*$ for $v\in V(\omega)$,
			\item the norms of the projections associated to the splitting are tempered with respect to $\sigma$, where a function $f:\Omega\to\real$ is called tempered if $\lim_{n\to\pm\infty}(1/n)\log|f(\sigma^n\omega)|=0$ for $\mathbb{P}$-a.e. $\omega$.
		\end{enumerate}
	
	\end{theorem}

	We call the above splitting \emph{Oseledets splitting} and the spaces $Y_j(\omega)$ \emph{Oseledets spaces}. The \emph{Oseledets filtration} $X=V_1(\omega)\supset\dots\supset V_l(\omega)\supset V_{l+1}(\omega)$ from Doan's theorem can be reconstructed via $V_{l+1}(\omega)=V(\omega)$ and 
	\begin{equation}
		V_k(\omega)=\bigoplus_{j=k}^lY_j(\omega)\oplus V(\omega)
	\end{equation}
	for $1\leq k\leq l$.\par 
	
	In \cref{subsec:Ginelli} we provide a method to compute the first $p\leq l$, $p<\infty$, Oseledets spaces for fixed $\omega$. The method requires cocycle data along the trajectory $\{\sigma^n\omega\ |\ n\in\integer\}$ and basic asymptotic properties that appear in the proofs of the METs from \cite{Doan.2010} and \cite{GonzalezTokmanQuas.2014}. That is, we need uniform upper bounds for asymptotics of the Oseledets filtration and uniform lower bounds for asymptotics of the Oseledets splitting. While bounds for the Oseledets filtration can be recovered from Doan's work \cite{Doan.2010}: 
	\begin{equation}\label[equation]{eqn:FiltrationUniformEstimateFT}
		\lim_{n\to\infty}\frac{1}{n}\log\|\mathcal{L}_{\omega}^{(n)}|_{V_j(\omega)}\|=\lambda_j
	\end{equation}
	for $1\leq j \leq l$ and
	\begin{equation}\label[equation]{eqn:FiltrationUniformEstimateFT2}
		\limsup_{n\to\infty}\frac{1}{n}\log\|\mathcal{L}_{\omega}^{(n)}|_{V(\omega)}\|\leq\kappa^*,
	\end{equation}
	bounds for the Oseledets splitting are due to \cite{GonzalezTokmanQuas.2014}. By choosing a suitable basis, González-Tokman and Quas reduce the cocycle along $Y_j$ to a cocycle of matrices (similar to \cite[lemma 19]{FroylandLloydQuas.2013}) for which uniform estimates are known. They arrive at \cite[lemma 2.14]{GonzalezTokmanQuas.2014} showing that, for every $\epsilon>0$ and for $\mathbb{P}$-a.e. $\omega$, there is a constant $c_j(\omega)>0$ such that  
	\begin{equation*}
		\|\mathcal{L}_\omega^{(n)}y\|\geq c_j(\omega)e^{n(\lambda_j-\epsilon)}
	\end{equation*}
	holds for every $n\geq 0$ and $y\in Y_j(\omega)\cap S$. By applying the same arguments to the sum of Oseledets spaces $Y_1(\omega)\oplus\dots \oplus Y_j(\omega)$, we get uniform lower bounds of growth rates inside sums of Oseledets spaces
	\begin{equation}\label[equation]{eqn:SplittingUniformEstimateFT}
		\liminf_{n\to\infty}\underset{y\in Y_1(\omega)\oplus\dots\oplus Y_j(\omega)\cap S}{\inf}\,\frac{1}{n}\log\|\mathcal{L}_\omega^{(n)}y\|\geq \lambda_j.
	\end{equation}
	\par
	
	In addition to the bounds for $\mathcal{L}_{\omega}^{(n)}$, we need similar bounds for $\mathcal{L}_{\sigma^{-n}\omega}^{(n)}$. Those can be obtained by applying \cite[lemma 8.2]{FroylandLloydQuas.2010} to the sequences of functions $(\log\|\mathcal{L}_\omega^{(n)}|_{V_j(\omega)}\|)_{n\in\natural}$ and $(\log\|\mathcal{L}_\omega^{(n)}|_{V(\omega)}\|)_{n\in\natural}$. We obtain
	\begin{equation}\label[equation]{eqn:FiltrationUniformEstimateBT}
		\lim_{n\to\infty}\frac{1}{n}\log\|\mathcal{L}_{\sigma^{-n}\omega}^{(n)}|_{V_j(\sigma^{-n}\omega)}\|=\lambda_j
	\end{equation}
	for $1\leq j \leq l$ and
	\begin{equation}\label[equation]{eqn:FiltrationUniformEstimateBT2}
		\limsup_{n\to\infty}\frac{1}{n}\log\|\mathcal{L}_{\sigma^{-n}\omega}^{(n)}|_{V(\sigma^{-n}\omega)}\|\leq\kappa^*
	\end{equation}
	for $\mathbb{P}$-a.e. $\omega$. Uniform lower bounds for the Oseledets splitting are again obtained from reduced systems via matrix cocycles (e.g., see proof of \cite[lemma 20]{FroylandLloydQuas.2013}). We have
	\begin{equation}\label[equation]{eqn:SplittingUniformEstimateBT}
		\liminf_{n\to\infty}\underset{y\in Y_1(\sigma^{-n}\omega)\oplus\dots\oplus Y_j(\sigma^{-n}\omega)\cap S}{\inf}\,\frac{1}{n}\log\|\mathcal{L}_{\sigma^{-n}\omega}^{(n)}y\|\geq \lambda_j.
	\end{equation}
	The uniform estimates for $\mathcal{L}_{\omega}^{(n)}$ and $\mathcal{L}_{\sigma^{-n}\omega}^{(n)}$ are then used in \cite{GonzalezTokmanQuas.2014} to prove temperedness of projections from \cref{thm:MET}.\par 
	
	Observe that $\ker\mathcal{L}_{\omega}^{(n)}\subset V(\omega)$ and $\ker\mathcal{L}_{\sigma^{-n}\omega}^{(n)}\subset V(\sigma^{-n}\omega)$ for every $n\in\natural$. Indeed, $\ker\mathcal{L}_{\omega}^{(n)}\subset V(\omega)$ follows from the different growth rates of the Oseledets splitting. Since $\ker\mathcal{L}_{\omega}^{(n)}\subset V(\omega)$ holds on a $\sigma$-invariant subset of $\Omega$, we get $\ker\mathcal{L}_{\sigma^{-n}\omega}^{(n)}\subset V(\sigma^{-n}\omega)$.\par
	
	Besides the uniform estimates, our convergence proof only needs the properties stated in \cref{thm:MET}. We remark that these properties are present in most versions of the MET that derive an Oseledets splitting. Hence, by adjusting the notation, \cref{sec:ComputingCLVs} can be generalized to various MET-scenarios. In particular, it generalizes the convergence proof from \cite{Noethen.2019}, which assumes the invertible two-sided MET found in \cite{Arnold.1998}.

\bigskip
	
	\section{Computing covariant Lyapunov vectors}\label{sec:ComputingCLVs}

	In this section we provide a method to compute \emph{covariant Lyapunov vectors} (CLVs). CLVs are a choice of basis vectors for Oseledets spaces that are normalized and covariant, meaning that CLVs at $\omega$ are mapped to CLVs at $\sigma^n\omega$ by $\mathcal{L}^{(n)}_\omega$ up to normalizing factors. According to the MET those factors have an exponential growth rate given by the associated Lyapunov exponents. Hence, CLVs describe asymptotic behavior of linear perturbations along trajectories. Using covariance and using that $\mathcal{L}$ is invertible on Oseledets spaces, we may push forward and backward CLVs at $\omega$ to obtain CLVs along the whole trajectory. Thus, the goal is to compute normalized basis vectors of Oseledets spaces at $\omega$.\par 
	
	Our method of choice is the Ginelli algorithm described in \cref{subsec:Ginelli}. It can be divided into two steps: one using $\mathcal{L}_{\sigma^{-n}\omega}^{(n)}$ to get an approximation of $Y_1(\omega)\oplus\dots\oplus Y_j(\omega)$ and one using $\mathcal{L}^{(n)}_\omega$ and $(\mathcal{L}^{(n)}_\omega)^{-1}$ (where it is defined) to extract an approximation of $Y_j(\omega)$ from the former approximation. We derive estimates involving forward propagation via $\mathcal{L}_{\sigma^{-n}\omega}^{(n)}$ and $\mathcal{L}^{(n)}_\omega$ in \cref{subsec:ForwardTime} and estimates involving backward propagation via $(\mathcal{L}^{(n)}_\omega)^{-1}$ in \cref{subsec:BackwardTime}. \cref{subsec:ConvergenceProof} combines those estimates to prove convergence of Ginelli's algorithm.

\subsection{Ginelli algorithm}\label{subsec:Ginelli}

	There are various algorithms to compute CLVs (see \cite{GinelliPoggiTurchiChateLiviPoliti.2007,GinelliChateLiviPoliti.2013,WolfeSamelson.2007,KuptsovParlitz.2012} or see \cite{FroylandHuelsMorrissWatson.2013} for a comparison). While they differ in their implementation, from an analytical point of view they rely either on computing a singular value decomposition of the cocycle (or its adjoint) or on pushing forward/backward a set of randomly chosen vectors. The first kind of methods can be analyzed directly using a technique due to Raghunathan \cite{Raghunathan.1979} (or see \cite{Arnold.1998}) that proves Oseledets' MET via a singular value decomposition of the cocycle. The second kind of methods can be analyzed using the different asymptotic growth rates associated to the Oseledets splitting. In particular, the latter may be used for METs on Banach spaces, like \cref{thm:MET}, where we may not have a singular value decomposition. Ginelli's algorithm \cite{GinelliPoggiTurchiChateLiviPoliti.2007} is part of the second class of methods.\par
	
	The fundamental idea behind Ginelli's algorithm is that almost every vector has a non-vanishing projection (subject to the Oseledets splitting) onto the first Oseledets space. Since vectors inside the first Oseledets space have the highest exponential growth rate, almost every vector will align with the first Oseledets space asymptotically in forward-time. Similarly, we expect the linear span of $k=m_1+\dots+m_j$ randomly chosen vectors to align with the fastest expanding $k$-dimensional subspace, the sum of the first $j$ Oseledets spaces, in forward-time. Reversing time, the fastest growing direction inside $Y_1\oplus\dots\oplus Y_j$ is the slowest growing direction in forward-time, i.e., the subspace $Y_j$. Thus, we have a means to compute Oseledets spaces.\par 
	
	At the level of Grassmannians, Ginelli's algorithm starts with a randomly chosen subspace $W\in\mathcal{G}_{m_1+\dots+m_j}(X)$, which is propagated from the far past to the present via $\mathcal{L}_{\sigma^{-n_1}\omega}^{(n_1)}$ to get an approximation of $Y_1(\omega)\oplus\dots\oplus Y_j(\omega)$ for large $n_1$. Then, $\mathcal{L}_{\sigma^{-n_1}\omega}^{(n_1)}W$ is propagated further via $\mathcal{L}_{\omega}^{(n_2)}$ to approximate $Y_1\oplus\dots\oplus Y_j$ in the far future. Next, the algorithm randomly chooses a subspace  $\tilde{W}\in\mathcal{G}_{m_j}(\mathcal{L}_{\sigma^{-n_1}\omega}^{(n_1+n_2)}W)$. This subspace is propagated backward to approximate $Y_j(\omega)$ for large $n_1,n_2$. (see \cref{fig:Ginelli})\par
	
	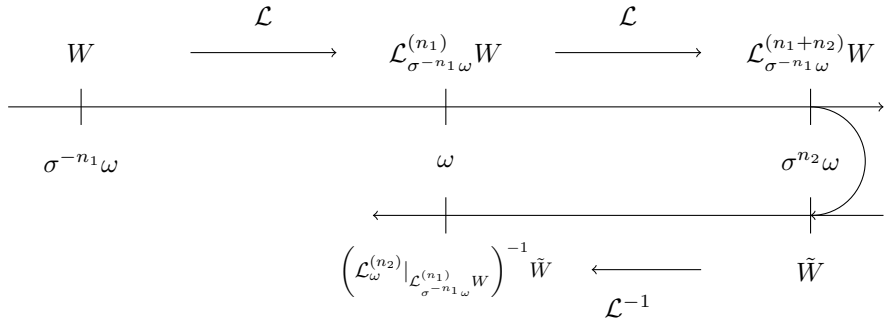
\begin{figure}\label[figure]{fig:Ginelli}
		
		\caption{Ginelli's algorithm at the level of Grassmannians}
		\begin{scaletikzpicturetowidth}{\textwidth*0.95}
			\begin{tikzpicture}[scale=\tikzscale]
			\node[] at (0,5) {};
			\node[] at (0,-5) {};
			
			\draw[->] (-12,1.5) -- (12,1.5);
			\draw[<-] (-2,-1.5) -- (12,-1.5);
			\draw[->] (10,1.5) arc (270:90:-1.5);
			
			\draw (-10,2) -- (-10,1);
			\draw (0,2) -- (0,1);
			\draw (10,2) -- (10,1);
			\draw (0,-2) -- (0,-1);
			\draw (10,-2) -- (10,-1);
			
			\draw[->] (-7,3) -- (-3,3);
			\draw[->] (3,3) -- (7,3);
			\draw[->] (7,-3) -- (4,-3);
			
			\node[] at (-5,4) {$\mathcal{L}$};
			\node[] at (5,4) {$\mathcal{L}$};
			\node[] at (5,-4) {$\mathcal{L}^{-1}$};
			
			\node[] at (-10,0) {$\sigma^{-n_1}\omega$};
			\node[] at (0,0) {$\omega$};
			\node[] at (10,0) {$\sigma^{n_2}\omega$};
			
			\node[] at (-10,3) {$W$};
			\node[] at (0,3) {$\mathcal{L}_{\sigma^{-n_1}\omega}^{(n_1)}W$};
			\node[] at (10,3) {$\mathcal{L}_{\sigma^{-n_1}\omega}^{(n_1+n_2)}W$};
			\node[] at (10,-3) {$\tilde{W}$};
			\node[] at (0,-3) {\scalebox{0.8}{$\left(\mathcal{L}_{\omega}^{(n_2)}|_{\mathcal{L}_{\sigma^{-n_1}\omega}^{(n_1)}W}\right)^{-1}\tilde{W}$}};
			\end{tikzpicture}
		\end{scaletikzpicturetowidth}
	
	\end{figure}

	In practice we express $W$ in terms of a basis $(x_1,\dots,x_k)$. By propagating these vectors, we can track the evolution of $W$. Similarly, we express $\tilde{W}$ in terms of a basis. The corresponding vectors can be described as coefficients of the propagated vectors of $W$. Hence, the backward propagation can be done solely inside a finite-dimensional coefficient space.\par 
	
	Let $X=H$ be a Hilbert space. To avoid that all vectors $x_1,\dots,x_k$ collapse onto the first Oseledets space, which renders them numerically indistinguishable, Ginelli et al. suggest to orthonormalize them between smaller propagation steps. While this procedure does not change the outcome of Ginelli's algorithm analytically, as the involved spaces remain the same, it helps with numerical stability. In particular, they use a $QR$-decomposition to store orthonormalized vectors in a matrix $Q$ and the cocycle on coefficient space in a matrix $R$ for each propagation step. The upper diagonal $R$-matrices can easily be inverted to perform the backward propagation in coefficient space. Using the identification, we substitute initial vectors for the backward propagation by an upper diagonal matrix representing their coefficients. For a more detailed description of the implementation in finite dimensions and examples see \cite{GinelliPoggiTurchiChateLiviPoliti.2007,GinelliChateLiviPoliti.2013}.\par 
	
	Taking the above into account, we define (the analytical kernel\footnote{We leave out numerical details of the implementation from \cite{GinelliPoggiTurchiChateLiviPoliti.2007}, since they do not affect the output of Ginelli's algorithm analytically.} of) Ginelli's algorithm on Hilbert spaces as
	\begin{equation}
		G_{\omega,k}^{n_1,n_2}:H^k\times\real^{k\times k}_{ru}\to H^k,
	\end{equation}
	where $\omega\in\Omega$ defines the trajectory, $k\leq m_1+\dots+m_p<\infty$ is the number of CLVs we wish to compute, $n_1,n_2\in\natural$ is the amount of steps needed along the past and the future of the trajectory, and $\real^{k\times k}_{ru}$ denotes the set of upper diagonal $k\times k$-matrices. $G_{\omega,k}^{n_1,n_2}$ operates on $((x_1,\dots,x_k),(r_{ij})_{i,j=1}^k)$ via the following steps:
	\begin{enumerate}
		\item forward propagation from $\sigma^{-n_1}\omega$ to $\omega$:
		\begin{equation*}
			\left(x_1^1,\dots,x_k^1\right):=\left(\mathcal{L}_{\sigma^{-n_1}\omega}^{(n_1)}x_1,\dots,\mathcal{L}_{\sigma^{-n_1}\omega}^{(n_1)}x_k\right).
		\end{equation*}
		\item forward propagation from $\omega$ to $\sigma^{n_2}\omega$:
		\begin{equation*}
			\left(x_1^2,\dots,x_k^2\right):=\left(\mathcal{L}_{\omega}^{(n_2)}x_1^1,\dots,\mathcal{L}_{\omega}^{(n_2)}x_k^1\right).
		\end{equation*}
		\item orthonormalizing (e.g., via Gram-Schmidt $GS$):
		\begin{equation*}
			\left(x_1^3,\dots,x_k^3\right):=GS\left(x_1^2,\dots,x_k^2\right).
		\end{equation*}
		\item initializing vectors for backward propagation:
		\begin{equation*}
			\left(y_1^1,y_2^1,\dots,y_k^1\right):=\left(r_{11}x_1^3,r_{12}x_1^3+r_{22}x_2^3,\dots,\sum_{i=1}^{k}r_{ik}x_i^3\right).
		\end{equation*}
		\item backward propagation from $\sigma^{n_2}\omega$ to $\omega$:
		\begin{equation*}
			\left(y_1^2,\dots,y_k^2\right):=\left(\left(\mathcal{L}_{\omega}^{(n_2)}|_{W^1}\right)^{-1}y_1^1,\dots,\left(\mathcal{L}_{\omega}^{(n_2)}|_{W^1}\right)^{-1}y_k^1\right),
		\end{equation*}
		where $W^1:=\textnormal{span}\left(x_1^1,\dots,x_k^1\right)$.
		\item normalizing:
		\begin{equation*}
			\left(y_1^3,\dots,y_k^3\right):=\left(\frac{y_1^2}{\|y_1^2\|},\dots,\frac{y_k^2}{\|y_k^2\|}\right).
		\end{equation*}
	\end{enumerate}
	We set $G_{\omega,k}^{n_1,n_2}((x_1,\dots,x_k),(r_{ij})_{i,j=1}^k):=\left(y_1^3,\dots,y_k^3\right)$ as our approximation of the first $k$ CLVs at $\omega$. The steps in computing $G_{\omega,k}^{n_1,n_2}$ are not tailored to the setting of \cref{thm:MET}, but rather can be performed using a sequence of operators describing forward and backward propagation. Therefore, the upcoming convergence theorem may be generalized to various MET-scenarios.\par 
	
	Before formulating the convergence theorem, we remark that whenever $G_{\omega,k+1}^{n_1,n_2}$ $((x_1,\dots,x_{k+1}),(r_{ij})_{i,j=1}^{k+1})$ is well-defined, its first $k$ components coincide with $G_{\omega,k}^{n_1,n_2}((x_1,\dots,x_{k}),(r_{ij})_{i,j=1}^{k})$. Thus, it suffices to investigate the case $k=m_1+\dots+m_p$ for finite $p\leq l$.\par 
	
	We group indices according to the multiplicities of Lyapunov exponents to simplify notation:
	\begin{equation}
		\left(x_{1_1},x_{1_2},\dots,x_{1_{m_1}},x_{2_1},\dots,x_{2_{m_2}},x_{3_1},\dots,x_{p_{m_p}}\right).
	\end{equation}
	
	\begin{theorem}[Convergence a.e. of Ginelli's algorithm]\label[theorem]{thm:Convergence}
		
		Let $\mathcal{R}=(\Omega,\mathcal{F},\mathbb{P},\sigma,H,$ $\mathcal{L})$ satisfy the assumptions of \cref{thm:MET} and let $k=m_1+\dots+m_p$ for some finite $p\leq l$. Moreover, set $\lambda_0:=\infty$ and $\lambda_{l+1}:=\kappa^*$.\par 
		
		On a subset $\Omega'\subset\Omega$ of full $\mathbb{P}$-measure, Ginelli's algorithm converges for almost every input. That is, fixing $\omega\in\Omega'$, for a.e. tuple $(x_1,\dots,x_k)\in H^k$, for a.e. $R\in\real^{k\times k}_{ru}$, and for all $j\leq p$, it holds
		\small
		\begin{multline}\label[equation]{eqn:ConvergenceProof}
			\limsup_{N\to\infty}\underset{n_1,n_2\geq N}{\sup}\,\frac{1}{\min(n_1,n_2)}\log\textnormal{d}_{\mathcal{G}}\left(\textnormal{span}\left\{\left.\left(G_{\omega,k}^{n_1,n_2}\right)_{j_i}\ \right|\ i=1,\dots,m_j\right\},Y_j(\omega)\right)\\
			\leq -\min\left(|\lambda_j-\lambda_{j-1}|,|\lambda_j-\lambda_{j+1}|\right)
		\end{multline} 
		\normalsize
		at $((x_1,\dots,x_{k}),R)$.\footnote{There are three concepts of ``almost every''. Firstly, the algorithm fixes $\omega$ from a set of full $\mathbb{P}$-measure to determine the trajectory along which Ginelli's algorithm shall be executed. Secondly and thirdly, the algorithm requires a tuple $(x_1,\dots,x_k)\in H^k$ and an upper diagonal matrix $R\in\real^{k\times k}_{ru}$ as inputs. ``A.e.'' with respect to the tuple is understood in terms of \textit{prevalence} \cite{OttYorke.2005}, whereas ``a.e.'' with respect to the matrix is meant in the usual Lebesgue sense. If $H$ is finite-dimensional, the two previous notions coincide.}
		
	\end{theorem}

	\cref{thm:Convergence} tells us that, generically, output vectors of Ginelli's algorithm, after grouping them according to the multiplicities of Lyapunov exponents, span subspaces that are exponentially close to the Oseledets spaces. Hence, the algorithm approximates CLVs. To get a good approximation, it is necessary to increase $n_1$ and $n_2$ simultaneously. In other words, the algorithm needs sufficient data along the past and the future of the trajectory. Moreover, \cref{thm:Convergence} reveals that the speed of convergence to the $j$-th Oseledets space $Y_j(\omega)$ is at least exponentially fast in proportion to the spectral gap between the associated Lyapunov exponent $\lambda_j$ and neighboring exponents.

\subsection{Forward-time estimates}\label{subsec:ForwardTime}

	During the next two subsections we assume that $X$ is a Banach space. Our first result investigates how certain subspaces evolve in the presence of an equivariant splitting under a given map. The estimates consist of terms that are well understood when the splitting is the Oseledets splitting.
	
	\begin{lemma}\label[lemma]{lemma:MapEstimateFT}
		
		Let $(Y,V),(Y',V')\in\textnormal{Comp}_k(X)$ be two pairs of closed complemented subspaces. Assume we have a bounded linear map $\mathcal{L}\in L(X)$ respecting the splittings, i.e., $\mathcal{L}Y\subset Y'$ and $\mathcal{L}V\subset V'$, such that $\ker \mathcal{L}\subset V$.\par
		
		If $W\in\mathcal{G}_k(X)$ is a complement of $V$ such that the degree of transversality satisfies
		\begin{equation}\label[equation]{eqn:MapEstimateFT}
			\underset{w\in W\cap S}{\inf}\,\textnormal{d}(w,V)\geq 2\|\Pi_{V||Y}\|\,\frac{\|\mathcal{L}|_V\|}{\inf_{y\in Y\cap S}\|\mathcal{L}y\|},
		\end{equation} 
		then
		\begin{equation}\label[equation]{eqn:MapEstimateFT2}
			\sup_{w'\in \mathcal{L}W\cap B}\,\textnormal{d}(w',Y'\cap B)\leq 4\,\frac{\|\Pi_{V||Y}\|}{\inf_{w\in W\cap S}\textnormal{d}(w,V)}\,\frac{\|\mathcal{L}|_V\|}{\inf_{y\in Y\cap S}\|\mathcal{L}y\|}.
		\end{equation}
		
	\end{lemma}

	\begin{proof}
		
		If $\mathcal{L}|_V=0$, then $\ker \mathcal{L}=V$. Thus, $\mathcal{L}$ restricts to an isomorphism between any complement $W$ of $V$ and $Y'$. In this case the claim is trivially satisfies.\par 
		
		Now, assume $\mathcal{L}|_V\neq 0$. Let $W$ be a complement as in the claim. For $w\in W\cap S$, it holds
		\begin{equation*}
			\|\mathcal{L}\Pi_{V||Y}w\|\leq\|\mathcal{L}|_V\|\,\|\Pi_{V||Y}\|
		\end{equation*}
		and
		\begin{align*}
			\|\mathcal{L}\Pi_{Y||V}w\|&\geq \underset{y\in Y\cap S}{\inf}\,\|\mathcal{L}y\|\,\|\Pi_{Y||V}w\|\\
			&=\underset{y\in Y\cap S}{\inf}\,\|\mathcal{L}y\|\,\|w-\Pi_{V||Y}w\|\\
			&\geq \underset{y\in Y\cap S}{\inf}\,\|\mathcal{L}y\|\,\textnormal{d}(w,V)\\
			&\geq 2\|\Pi_{V||Y}\|\,\|\mathcal{L}|_V\|>0.
		\end{align*}
		Combining both estimates, we get
		\begin{equation}\label[equation]{eqn:MapEstimateFT3}
			\frac{\|\mathcal{L}\Pi_{V||Y}w\|}{\|\mathcal{L}\Pi_{Y||V}w\|}\leq \frac{1}{2}.
		\end{equation}
		To derive \cref{eqn:MapEstimateFT2}, it is enough to estimate $\textnormal{d}(\mathcal{L}w/\|\mathcal{L}w\|,Y'\cap B)$ for $w\in W\cap S$. Write $w=y+v$ according to the decomposition $X=Y\oplus V$. We have
		\begin{align*}
			\textnormal{d}\left(\frac{\mathcal{L}w}{\|\mathcal{L}w\|},Y'\cap B\right)&\leq \left\|\frac{\mathcal{L}w}{\|\mathcal{L}w\|}-\frac{\mathcal{L}y}{\|\mathcal{L}y\|}\right\|\\
			&=\left\|\frac{\mathcal{L}v}{\|\mathcal{L}w\|}-\left(\frac{1}{\|\mathcal{L}y\|}-\frac{1}{\|\mathcal{L}w\|}\right)\mathcal{L}y\right\|\\
			&\leq\frac{\|\mathcal{L}v\|}{\|\mathcal{L}w\|}+\left|1-\frac{\|\mathcal{L}y\|}{\|\mathcal{L}w\|}\right|.
		\end{align*}
		Since $y\neq 0$ and by \cref{eqn:MapEstimateFT3}, we can estimate the first term:
		\begin{equation*}
			\frac{\|\mathcal{L}v\|}{\|\mathcal{L}w\|}\leq \frac{\|\mathcal{L}v\|}{\|\mathcal{L}y\|-\|\mathcal{L}v\|}=\frac{\|\mathcal{L}v\|}{\|\mathcal{L}y\|}\left(1-\frac{\|\mathcal{L}v\|}{\|\mathcal{L}y\|}\right)^{-1}\leq 2\frac{\|\mathcal{L}v\|}{\|\mathcal{L}y\|}.
		\end{equation*}
		For the other term, we distinguish between two cases. If $\|\mathcal{L}y\|/\|\mathcal{L}w\|\leq 1$, then
		\begin{equation*}
			1-\frac{\|\mathcal{L}y\|}{\|\mathcal{L}w\|}\leq 1-\frac{\|\mathcal{L}y\|}{\|\mathcal{L}y\|+\|\mathcal{L}v\|}=\frac{\|\mathcal{L}v\|}{\|\mathcal{L}y\|+\|\mathcal{L}v\|}\leq \frac{\|\mathcal{L}v\|}{\|\mathcal{L}y\|}.
		\end{equation*}
		If $\|\mathcal{L}y\|/\|\mathcal{L}w\|\geq 1$, then
		\begin{equation*}
			\frac{\|\mathcal{L}y\|}{\|\mathcal{L}w\|}-1=\frac{\|\mathcal{L}y\|-\|\mathcal{L}w\|}{\|\mathcal{L}w\|}\leq\frac{\|\mathcal{L}v\|}{\|\mathcal{L}w\|}\leq 2\frac{\|\mathcal{L}v\|}{\|\mathcal{L}y\|}.
		\end{equation*}
		In total, we get
		\begin{equation*}
			\textnormal{d}\left(\frac{\mathcal{L}w}{\|\mathcal{L}w\|},Y'\cap B\right)\leq 4\frac{\|\mathcal{L}v\|}{\|\mathcal{L}y\|}.
		\end{equation*}
		Since $v=\Pi_{V||Y}w$ and $y=\Pi_{Y||V}w$, the claim follows from the estimates in the beginning.
		
	\end{proof}

	\begin{corollary}\label[corollary]{cor:MapEstimateFT}
		
		In the setting of \cref{lemma:MapEstimateFT}, it holds
		\begin{equation}
			\|\Pi_{V'||Y'}|_{\mathcal{L}W}\|\leq 2\,\frac{\|\Pi_{V||Y}\|}{\inf_{w\in W\cap S}\textnormal{d}(w,V)}\,\frac{\|\mathcal{L}|_V\|}{\inf_{y\in Y\cap S}\|\mathcal{L}y\|}.
		\end{equation}
		
	\end{corollary}

	\begin{proof}
		
		The corollary follows from 
		\begin{equation*}
			\|\Pi_{V'||Y'}|_{\mathcal{L}W}\|=\underset{w\in W\cap S}{\sup}\,\left\|\Pi_{V'||Y'}\frac{\mathcal{L}w}{\|\mathcal{L}w\|}\right\|=\underset{w\in W\cap S}{\sup}\,\frac{\|\mathcal{L}\Pi_{V||Y}w\|}{\|\mathcal{L}w\|}
		\end{equation*}
		and the estimate of $\|\mathcal{L}v\|/\|\mathcal{L}w\|$ in the proof of \cref{lemma:MapEstimateFT}.
		
	\end{proof}

	Next, we derive two lemmata that handle sequences of maps acting on equivariant splittings with different asymptotic growth rates. The first lemma is concerned with propagation from present to future states, whereas the second lemma treats propagation from the past to the present.
	
	\begin{lemma}\label[lemma]{lemma:SequenceEstimateFT}
		
		Let $(Y,V)\in\textnormal{Comp}_k(X)$ and $(Y(n),V(n))\in\textnormal{Comp}_k(X)$ for $n\in\natural$. Assume we have bounded linear maps $\mathcal{L}(n)\in L(X)$ respecting the splittings, i.e., $\mathcal{L}(n)Y\subset Y(n)$ and $\mathcal{L}(n)V\subset V(n)$, such that $\ker \mathcal{L}(n)\subset V$. Furthermore, assume there are numbers $\infty>\lambda_Y>\lambda_V\geq-\infty$ such that
		\begin{equation*}
			\limsup_{n\to\infty}\frac{1}{n}\log\|\mathcal{L}(n)|_V\|\leq \lambda_V
		\end{equation*}
		and
		\begin{equation*}
			\liminf_{n\to\infty}\underset{y\in Y\cap S}{\inf}\,\frac{1}{n}\log\|\mathcal{L}(n)y\|\geq \lambda_Y.
		\end{equation*}
		\par
		
		Then, we have
		\begin{equation}
			\limsup_{n\to\infty}\frac{1}{n}\log\textnormal{d}_{\mathcal{G}}\left(\mathcal{L}(n)W,Y(n)\right)\leq -|\lambda_Y-\lambda_V|
		\end{equation}
		for any complement $W$ of $V$.
		
	\end{lemma}

	\begin{proof}
		
		According to the assumptions we have
		\begin{equation*}
			\limsup_{n\to\infty}\frac{1}{n}\log\frac{\|\mathcal{L}(n)|_V\|}{\inf_{y\in Y\cap S}\|\mathcal{L}(n)y\|}\leq -|\lambda_Y-\lambda_V|<0,
		\end{equation*}
		i.e., the quotient $\|\mathcal{L}(n)|_V\|/(\inf_{y\in Y\cap S}\|\mathcal{L}(n)y\|)$ decays exponentially fast with $n$. Thus, for any complement $W$ of $V$, there is $N>0$ such that \cref{eqn:MapEstimateFT} of \cref{lemma:MapEstimateFT} is satisfied for all $n\geq N$. Applying the lemma, we get
		\begin{equation*}
			\limsup_{n\to\infty}\frac{1}{n}\log \sup_{w'\in \mathcal{L}(n)W\cap B}\,\textnormal{d}(w',Y(n)\cap B)\leq -|\lambda_Y-\lambda_V|.
		\end{equation*}
		The claim follows from \cref{lemma:SymmetryOfCloseness}.
		
	\end{proof}

	\cref{lemma:SequenceEstimateFT} implies that complements of spaces of the Oseledets filtration will align with Oseledets spaces asymptotically (at an exponential speed). Moreover, the lemma tells us that any two complements of $V$ will align asymptotically if they have a uniformly higher growth rate than $V$. Interestingly, we do not need the existence of an Oseledets splitting. In fact, the lemma may be applied to systems with a possibly non-invertible base (e.g, see \cite[theorem 2]{BarreiraSilva.2005} or \cite{Blumenthal.2015}).
	
	\begin{lemma}\label[lemma]{lemma:SequenceEstimateFT2}
		
		Let $(Y,V)\in\textnormal{Comp}_k(X)$ and $(Y(-n),V(-n))\in\textnormal{Comp}_k(X)$ for $n\in\natural$. Assume we have bounded linear maps $\mathcal{L}(-n)\in L(X)$ respecting the splittings, i.e., $\mathcal{L}(-n)Y(-n)\subset Y$ and $\mathcal{L}(-n)V(-n)\subset V$, such that $\ker \mathcal{L}(-n)\subset V(-n)$. Furthermore, assume that
		\begin{equation*}
			\lim_{n\to\infty}\frac{1}{n}\log\|\Pi_{V(-n)||Y(-n)}\|=0
		\end{equation*}
		and that there are numbers $\infty>\lambda_Y>\lambda_V\geq-\infty$ such that
		\begin{equation*}
			\limsup_{n\to\infty}\frac{1}{n}\log\|\mathcal{L}(-n)|_{V(-n)}\|\leq \lambda_V
		\end{equation*}
		and
		\begin{equation*}
			\liminf_{n\to\infty}\underset{y\in Y(-n)\cap S}{\inf}\,\frac{1}{n}\log\|\mathcal{L}(-n)y\|\geq \lambda_Y.
		\end{equation*}
		\par 
		
		Then, we have
		\begin{equation}
			\limsup_{n\to\infty}\frac{1}{n}\log\textnormal{d}_{\mathcal{G}}\left(\mathcal{L}(-n)W,Y\right)\leq -|\lambda_Y-\lambda_V|
		\end{equation}
		for any well-separated common complement $W$ of $(V(-n))_{n\in\natural}$.
		
	\end{lemma}

	\begin{proof}
		
		As in \cref{lemma:SequenceEstimateFT}, we see that 
		\begin{equation*}
			\limsup_{n\to\infty}\frac{1}{n}\log\frac{\|\mathcal{L}(-n)|_{V(-n)}\|}{\inf_{y\in Y(-n)\cap S}\|\mathcal{L}(-n)y\|}\leq -|\lambda_Y-\lambda_V|.
		\end{equation*}
		By our assumption on the growth rates of the associated projections, we get
		\begin{multline*}
			\limsup_{n\to\infty}\frac{1}{n}\log \left(2\|\Pi_{V(-n)||Y(-n)}\|\,\frac{\|\mathcal{L}(-n)|_{V(-n)}\|}{\inf_{y\in Y(-n)\cap S}\|\mathcal{L}(-n)y\|}\right)\\
			\leq -|\lambda_Y-\lambda_V|<0.
		\end{multline*}
		In particular, by \cref{def:WellSeparating} any well-separated common complement of $(V(-n))_{n\in\natural}$ fulfills \cref{eqn:MapEstimateFT} for $n$ large enough. The claim may be derived as in the proof of \cref{lemma:SequenceEstimateFT}.
		
	\end{proof}

	\begin{corollary}\label[corollary]{cor:SequenceEstimateFT}
		
		In the setting of \cref{lemma:SequenceEstimateFT2}, we have
		\begin{equation}
			\limsup_{n\to\infty}\frac{1}{n}\log\|\Pi_{V||Y}|_{\mathcal{L}(-n)W}\|\leq -|\lambda_Y-\lambda_V|
		\end{equation}
		for any well-separated common complement $W$ of $(V(-n))_{n\in\natural}$.
		
	\end{corollary}

	\begin{proof}
		
		Since \cref{lemma:MapEstimateFT} and \cref{cor:MapEstimateFT} give the same estimate up to a factor of $2$, the proof of \cref{cor:SequenceEstimateFT} is the same as for \cref{lemma:SequenceEstimateFT2}.
		
	\end{proof}

	\begin{remark}
		
		With additional assumptions on growth rates, the requirement on $\Pi_{V(-n)||Y(-n)}$ in \cref{lemma:SequenceEstimateFT2} may be derived from growth rates as in the proof of \cref{thm:MET} in \cite{GonzalezTokmanQuas.2014}.
		
	\end{remark}

	The following theorem gives us convergence of certain subspaces in Banach spaces to the sum of the first Oseledets spaces in forward-time.
	
	\begin{theorem}\label[theorem]{thm:ConvergenceBSFT}
		
		Let $\mathcal{R}$ be as in \cref{thm:MET} and $\omega\in\Omega$ such that the Oseledets splitting exists. Write $\lambda_{l+1}:=\kappa^*$ and fix some finite $j\leq l$.\par 
		
		If \cref{eqn:FiltrationUniformEstimateFT,eqn:FiltrationUniformEstimateFT2,eqn:SplittingUniformEstimateFT} hold\footnote{We remark that \cref{eqn:FiltrationUniformEstimateFT,eqn:FiltrationUniformEstimateFT2,eqn:SplittingUniformEstimateFT} and \cref{eqn:FiltrationUniformEstimateBT,eqn:FiltrationUniformEstimateBT2,eqn:SplittingUniformEstimateBT} hold for $\mathbb{P}$-a.e. $\omega\in\Omega$.}, then 
		\begin{equation}
			\limsup_{n\to\infty}\frac{1}{n}\log\textnormal{d}_{\mathcal{G}}\left(\mathcal{L}_{\omega}^{(n)}W,Y_1(\sigma^n\omega)\oplus\dots\oplus Y_j(\sigma^n\omega)\right)\leq -|\lambda_j-\lambda_{j+1}|
		\end{equation}
		for any complement $W$ of $V_{j+1}(\omega)$.\par 
		
		If \cref{eqn:FiltrationUniformEstimateBT,eqn:FiltrationUniformEstimateBT2,eqn:SplittingUniformEstimateBT} hold, then 
		\begin{equation}
			\limsup_{n\to\infty}\frac{1}{n}\log\textnormal{d}_{\mathcal{G}}\left(\mathcal{L}_{\sigma^{-n}\omega}^{(n)}W,Y_1(\omega)\oplus\dots\oplus Y_j(\omega)\right)\leq -|\lambda_j-\lambda_{j+1}|
		\end{equation}
		for any well-separating common complement $W$ of $(V_{j+1}(\sigma^{-n}\omega))_{n\in\natural}$.
		
	\end{theorem}

	\begin{proof}
		
		The proof is a direct application of \cref{lemma:SequenceEstimateFT} and \cref{lemma:SequenceEstimateFT2} to the splittings $(Y,V)=(Y_1(\omega)\oplus\dots\oplus Y_j(\omega),V_{j+1}(\omega))$, $(Y(n),V(n)):=(Y_1(\sigma^n\omega)\oplus\dots\oplus Y_j(\sigma^n\omega),V_{j+1}(\sigma^n\omega))$ for $n\in\integer$, and to the maps $\mathcal{L}(n):=\mathcal{L}_{\omega}^{(n)}$ and $\mathcal{L}(-n):=\mathcal{L}_{\sigma^{-n}\omega}^{(n)}$ for $n\in\natural$.
		
	\end{proof}

	In view of \cref{thm:WellSeparating}, \cref{thm:ConvergenceBSFT} for Hilbert spaces implies that we can compute the sum of the first Oseledets spaces $Y_1\oplus\dots\oplus Y_j$ at $\omega$ or asymptotically by pushing forward a set of $m_1+\dots+m_j$ randomly chosen vectors. The convergence is exponentially fast with a rate given by the gap between the consecutive Lyapunov exponents $\lambda_j$ and $\lambda_{j+1}$.

\subsection{Backward-time estimates}\label{subsec:BackwardTime}

	In this subsection we investigate backward propagation of certain subspaces. Since we did not assume an invertible cocycle, we cannot simply apply our results from \cref{subsec:ForwardTime} to an inverted system, as it is done in \cite{Noethen.2019}. Instead, we use growth rates in forward-time to deduce properties for backward propagation along sequences of subspaces obtained in \cref{thm:ConvergenceBSFT}.
	
	\begin{lemma}\label[lemma]{lemma:MapEstimateBT}
		
		Let $(Y_1,V_1)\in \textnormal{Comp}_{k_1}(X)$ and $(Y_2,V_2)\in\textnormal{Comp}_{k_2}(V_1)$, so that $X=Y_1\oplus V_1$ and $V_1=Y_2\oplus V_2$. Moreover, let $W_i$ be a complement of $V_i$ in $X$ for $i=1,2$ such that $W_1\subset W_2$. Assume we have a map $\mathcal{L}\in L(X)$ with $\ker\mathcal{L}\subset V_2$.\par
		 
		If $\tilde{W}\in\mathcal{G}_{k_2}(W_2)$ is a complement of $W_1$ in $W_2$ and if $\tilde{w}\in\tilde{W}\cap S$ such that
		\begin{multline}\label[equation]{eqn:MapEstimateBT}
			\textnormal{d}(\tilde{w},Y_2)\geq\\
			(2\|\Pi_{V_1||W_1}\|+\|\Pi_{V_1||Y_1}\|\,\|\Pi_{W_1||V_1}\|)\frac{\|\mathcal{L}|_{V_1}\|}{\inf_{y\in Y_1\cap S}\|\mathcal{L}y\|}+\|\Pi_{V_2||Y_1\oplus Y_2}|_{W_2}\|,
		\end{multline}
		then
		\begin{multline}\label[equation]{eqn:MapEstimateBT2}
			\textnormal{d}\left(\frac{\mathcal{L}\tilde{w}}{\|\mathcal{L}\tilde{w}\|},\mathcal{L}W_1\right)\leq\\
			\frac{2\|\mathcal{L}|_{V_1}\|\,\|\Pi_{V_1||W_1}\|}{\left(\inf_{y\in Y_1\cap S}\|\mathcal{L}y\|\right)\left(\textnormal{d}(\tilde{w},Y_2)-\|\Pi_{V_2||Y_1\oplus Y_2}|_{W_2}\|\right)-\|\mathcal{L}|_{V_1}\|\,\|\Pi_{V_1||Y_1}\|\,\|\Pi_{W_1||V_1}\|}.
		\end{multline}
		
	\end{lemma}

	\begin{proof}
		
		Since $V_1=Y_2\oplus V_2$ is a splitting with $Y_2\neq\{0\}$ and $\ker\mathcal{L}\subset V_2$, it holds $\mathcal{L}|_{V_1}\neq 0$.\par 
		
		Let $\tilde{w}\in\tilde{W}\cap S$ be as in the claim, so that \cref{eqn:MapEstimateBT} is satisfied. We estimate
		\begin{equation*}
			\|\mathcal{L}\Pi_{V_1||W_1}\tilde{w}\|\leq\|\mathcal{L}|_{V_1}\|\,\|\Pi_{V_1||W_1}\|
		\end{equation*}
		and
		\begin{align*}
			\|\mathcal{L}\Pi_{W_1||V_1}\tilde{w}\|&=\|\mathcal{L}(\Pi_{Y_1||V_1}+\Pi_{V_1||Y_1})\Pi_{W_1||V_1}\tilde{w}\|\\
			&\geq \|\mathcal{L}\Pi_{Y_1||V_1}\Pi_{W_1||V_1}\tilde{w}\|-\|\mathcal{L}\Pi_{V_1||Y_1}\Pi_{W_1||V_1}\tilde{w}\|\\
			&\geq \left(\underset{y\in Y_1\cap S}{\inf}\,\|\mathcal{L}y\|\right)\|\Pi_{Y_1||V_1}\Pi_{W_1||V_1}\tilde{w}\|-\|\mathcal{L}|_{V_1}\|\,\|\Pi_{V_1||Y_1}\|\,\|\Pi_{W_1||V_1}\|.
		\end{align*}
		The term with two consecutive projections applied to $\tilde{w}$ can be estimates further via
		\begin{align*}
			\|\Pi_{Y_1||V_1}\Pi_{W_1||V_1}\tilde{w}\|&=\|\Pi_{W_1||V_1}\tilde{w}-\Pi_{V_1||Y_1}\Pi_{W_1||V_1}\tilde{w}\|\\
			&=\|\tilde{w}-\Pi_{V_1||W_1}\tilde{w}-\Pi_{V_1||Y_1}\Pi_{W_1||V_1}\tilde{w}\|\\
			&=\|\tilde{w}-\Pi_{V_1||Y_1}(\Pi_{V_1||W_1}+\Pi_{W_1||V_1})\tilde{w}\|\\
			&=\|\tilde{w}-\Pi_{V_1||Y_1}\tilde{w}\|\\
			&=\|\tilde{w}-\Pi_{Y_2||V_2}\Pi_{V_1||Y_1}\tilde{w}-\Pi_{V_2||Y_2}\Pi_{V_1||Y_1}\tilde{w}\|\\
			&\geq\|\tilde{w}-\Pi_{Y_2||V_2}\Pi_{V_1||Y_1}\tilde{w}\|-\|\Pi_{V_2||Y_2}\Pi_{V_1||Y_1}\tilde{w}\|\\
			&\geq\textnormal{d}(\tilde{w},Y_2)-\|\Pi_{V_2||Y_2}\Pi_{V_1||Y_1}\tilde{w}\|\\
			&=\textnormal{d}(\tilde{w},Y_2)-\|\Pi_{V_2||Y_1\oplus Y_2}\tilde{w}\|\\
			&\geq\textnormal{d}(\tilde{w},Y_2)-\|\Pi_{V_2||Y_1\oplus Y_2}|_{W_2}\|.
		\end{align*}
		Note that $\Pi_{V_2||Y_2}$ and $\Pi_{Y_2||V_2}$ are projections defined on $V_1$. By \cref{eqn:MapEstimateBT} we have
		\begin{align*}
			\|\Pi_{Y_1||V_1}\Pi_{W_1||V_1}\tilde{w}\|\geq(2\|\Pi_{V_1||W_1}\|+\|\Pi_{V_1||Y_1}\|\,\|\Pi_{W_1||V_1}\|)\frac{\|\mathcal{L}|_{V_1}\|}{\inf_{y\in Y_1\cap S}\|\mathcal{L}y\|}.
		\end{align*}
		Hence, we get
		\begin{equation*}
			\|\mathcal{L}\Pi_{W_1||V_1}\tilde{w}\|\geq 2\|\Pi_{V_1||W_1}\|\,\|\mathcal{L}|_{V_1}\|>0
		\end{equation*}
		and
		\begin{equation*}
			\frac{\|\mathcal{L}\Pi_{V_1||W_1}\tilde{w}\|}{\|\mathcal{L}\Pi_{W_1||V_1}\tilde{w}\|}\leq\frac{1}{2}.
		\end{equation*}
		Finally, it holds
		\begin{align*}
			\textnormal{d}\left(\frac{\mathcal{L}\tilde{w}}{\|\mathcal{L}\tilde{w}\|},\mathcal{L}W_1\right)&\leq\left\|\frac{\mathcal{L}\tilde{w}}{\|\mathcal{L}\tilde{w}\|}-\frac{\mathcal{L}\Pi_{W_1||V_1}\tilde{w}}{\|\mathcal{L}\tilde{w}\|}\right\|\\
			&=\frac{\|\mathcal{L}\Pi_{V_1||W_1}\tilde{w}\|}{\|\mathcal{L}\tilde{w}\|}\\
			&\leq\frac{\|\mathcal{L}\Pi_{V_1||W_1}\tilde{w}\|}{\|\mathcal{L}\Pi_{W_1||V_1}\tilde{w}\|-\|\mathcal{L}\Pi_{V_1||W_1}\tilde{w}\|}\\
			&=\frac{\|\mathcal{L}\Pi_{V_1||W_1}\tilde{w}\|}{\|\mathcal{L}\Pi_{W_1||V_1}\tilde{w}\|}\left(1-\frac{\|\mathcal{L}\Pi_{V_1||W_1}\tilde{w}\|}{\|\mathcal{L}\Pi_{W_1||V_1}\tilde{w}\|}\right)^{-1}\\
			&\leq 2\frac{\|\mathcal{L}\Pi_{V_1||W_1}\tilde{w}\|}{\|\mathcal{L}\Pi_{W_1||V_1}\tilde{w}\|}.
		\end{align*}
		Estimating the numerator and denominator as in the beginning of the proof, we arrive at \cref{eqn:MapEstimateBT2}. 
		
	\end{proof}

	\begin{corollary}\label[corollary]{cor:MapEstimateBT}
		
		Let $Y_i,V_i,W_i$ for $i=1,2$ and $\mathcal{L}$ be as in \cref{lemma:MapEstimateBT}.\par 
		
		If $\tilde{W}\subset W_2$ is a complement of $W_1$ in $W_2$ satisfying
		\begin{equation}
			\underset{\tilde{w}'\in\mathcal{L}\tilde{W}\cap S}{\inf}\,\textnormal{d}(\tilde{w}',\mathcal{L}W_1)\geq\delta
		\end{equation}
		for some $0<\delta\leq 1$, then
		\begin{multline}
			\underset{\tilde{w}\in\tilde{W}\cap B}{\sup}\,\textnormal{d}(\tilde{w},Y_2\cap B)\leq\\
			2\left(\frac{2}{\delta}\|\Pi_{V_1||W_1}\|+\|\Pi_{V_1||Y_1}\|\,\|\Pi_{W_1||V_1}\|\right)\frac{\|\mathcal{L}|_{V_1}\|}{\inf_{y\in Y_1\cap S}\|\mathcal{L}y\|}+2\|\Pi_{V_2||Y_1\oplus Y_2}|_{W_2}\|.
		\end{multline}
		
	\end{corollary}

	\begin{proof}
		
		Assume $\tilde{w}\in\tilde{W}\cap S$ fulfills
		\begin{multline*}
			\textnormal{d}(\tilde{w},Y_2)>\\
			\left(\frac{2}{\delta}\|\Pi_{V_1||W_1}\|+\|\Pi_{V_1||Y_1}\|\,\|\Pi_{W_1||V_1}\|\right)\frac{\|\mathcal{L}|_{V_1}\|}{\inf_{y\in Y_1\cap S}\|\mathcal{L}y\|}+\|\Pi_{V_2||Y_1\oplus Y_2}|_{W_2}\|,
		\end{multline*}
		then by \cref{lemma:MapEstimateBT}
		\begin{multline*}
			\delta\leq\textnormal{d}\left(\frac{\mathcal{L}\tilde{w}}{\|\mathcal{L}\tilde{w}\|},\mathcal{L}W_1\right)\leq\\
			\frac{2\|\mathcal{L}|_{V_1}\|\,\|\Pi_{V_1||W_1}\|}{\left(\inf_{y\in Y_1\cap S}\|\mathcal{L}y\|\right)\left(\textnormal{d}(\tilde{w},Y_2)-\|\Pi_{V_2||Y_1\oplus Y_2}|_{W_2}\|\right)-\|\mathcal{L}|_{V_1}\|\,\|\Pi_{V_1||Y_1}\|\,\|\Pi_{W_1||V_1}\|}.
		\end{multline*}
		However, the former would be strictly smaller than $\delta$ by our assumption on $\textnormal{d}(\tilde{w},Y_2)$. Hence, we must have
		\begin{multline*}
			\underset{\tilde{w}\in\tilde{W}\cap S}{\sup}\,\textnormal{d}(\tilde{w},Y_2)\leq\\
			\left(\frac{2}{\delta}\|\Pi_{V_1||W_1}\|+\|\Pi_{V_1||Y_1}\|\,\|\Pi_{W_1||V_1}\|\right)\frac{\|\mathcal{L}|_{V_1}\|}{\inf_{y\in Y_1\cap S}\|\mathcal{L}y\|}+\|\Pi_{V_2||Y_1\oplus Y_2}|_{W_2}\|.
		\end{multline*}
		By \cite[chapter IV, §2.1]{Kato.1995} it holds
		\begin{equation*}
			\underset{\tilde{w}\in\tilde{W}\cap S}{\sup}\,\textnormal{d}(\tilde{w},Y_2\cap S)\leq 2\underset{\tilde{w}\in\tilde{W}\cap S}{\sup}\,\textnormal{d}(\tilde{w},Y_2).
		\end{equation*}
		Since
		\begin{equation*}
			\underset{\tilde{w}\in\tilde{W}\cap B}{\sup}\,\textnormal{d}(\tilde{w},Y_2\cap B)\leq \underset{\tilde{w}\in\tilde{W}\cap S}{\sup}\,\textnormal{d}(\tilde{w},Y_2\cap S),
		\end{equation*}
		the claim follows.	
			
	\end{proof}

	From \cref{cor:MapEstimateBT} we can derive an upper bound on the distance between $\tilde{W}$ and $Y_2$ from a lower bound on the degree of transversality of $(\mathcal{L}\tilde{W},\mathcal{L}W_1)$. In that sense, the corollary is similar to \cref{lemma:MapEstimateFT} but with backward propagation.\par 
	
	Next, we use the spaces $W_1,W_2$ to connect estimates from \cref{subsec:ForwardTime} to backward propagation, ultimately giving us an understanding of Ginelli's algorithm at the level of maps.
	
	\begin{lemma}\label[lemma]{lemma:SequenceEstimateBT}
		
		Let $(Y_1,V_1)\in\textnormal{Comp}_{k_1}(X)$, $(Y_2,V_2)\in\textnormal{Comp}_{k_2}(V_1)$, and $\infty>\lambda_{Y_1}>\lambda_{V_1}=\lambda_{Y_2}>\lambda_{V_2}\geq-\infty$.\par 
		
		For the past data, let $(Y_1(-n),V_1(-n))\in\textnormal{Comp}_{k_1}(X)$ and $(Y_2(-n),V_2(-n))$ $\in\textnormal{Comp}_{k_2}(V_1(-n))$ for $n\in\natural$. Assume we have bounded linear maps $\mathcal{L}(-n)\in L(X)$ respecting the splittings , i.e., $\mathcal{L}(-n)Y_i(-n)\subset Y_i$ for $i=1,2$ and $\mathcal{L}(-n)$ $V_2(-n)\subset V_2$, such that $\ker\mathcal{L}(-n)\subset V_2(-n)$ for $n\in\natural$. Moreover, assume that
		\begin{enumerate}
			\item $\lim_{n\to\infty}(1/n)\log\|\Pi_{V_1(-n)||Y_1(-n)}\|=0$,
			\item $\lim_{n\to\infty}(1/n)\log\|\Pi_{V_2(-n)||Y_1(-n)\oplus Y_2(-n)}\|=0$,
			\item $\limsup_{n\to\infty}(1/n)\log\|\mathcal{L}(-n)|_{V_i(-n)}\|\leq\lambda_{V_i}$ for $i=1,2$,
			\item $\liminf_{n\to\infty}\inf_{y\in Y_1(-n)\cap S}(1/n)\log\|\mathcal{L}(-n)y\|\geq \lambda_{Y_1}$, and
			\item $\liminf_{n\to\infty}\inf_{y\in Y_1(-n)\oplus Y_2(-n)\cap S}(1/n)\log\|\mathcal{L}(-n)y\|\geq \lambda_{Y_2}$.
		\end{enumerate}\par
	
		For the future data, let $(Y_1(n),V_1(n))\in\textnormal{Comp}_{k_1}(X)$ and $(Y_2(n),V_2(n))\in\textnormal{Comp}_{k_2}(V_1(n))$ for $n\in\natural$. Assume we have bounded linear maps $\mathcal{L}(n)\in L(X)$ respecting the splittings , i.e., $\mathcal{L}(n)Y_i\subset Y_i(n)$ for $i=1,2$ and $\mathcal{L}(n)V_2\subset V_2(n)$, such that $\ker\mathcal{L}(n)\subset V_2$ for $n\in\natural$. Moreover, assume that
		\begin{enumerate}
			\item[6.] $\limsup_{n\to\infty}(1/n)\log\|\mathcal{L}(n)|_{V_1}\|\leq\lambda_{V_1}$ and
			\item[7.] $\liminf_{n\to\infty}\inf_{y\in Y_1\cap S}(1/n)\log\|\mathcal{L}(n)y\|\geq \lambda_{Y_1}$.
		\end{enumerate}\par
	
		Let $W_i$ be a well-separating common complement of $(V_i(-n))_{n\in\natural}$ for $i=1,2$ such that $W_1\subset W_2$. If $(\tilde{W}(n_1,n_2))_{n_1,n_2\in\natural}$ is a family of subspaces such that $\mathcal{L}(n_2)\mathcal{L}(-n_1)W_1\oplus\tilde{W}(n_1,n_2)=\mathcal{L}(n_2)\mathcal{L}(-n_1)W_2$, and if
		\begin{equation}\label[equation]{eqn:SequenceEstimateBT}
			\underset{\tilde{w}\in\tilde{W}(n_1,n_2)\cap S}{\inf}\,\textnormal{d}(\tilde{w},\mathcal{L}(n_2)\mathcal{L}(-n_1)W_1)\geq\delta
		\end{equation}
		for some constant $0<\delta\leq 1$, then
		\begin{multline}\label[equation]{eqn:SequenceEstimateBT2}
			\limsup_{N\to\infty}\underset{n_1,n_2\geq N}{\sup}\,\frac{1}{\min(n_1,n_2)}\log\textnormal{d}_{\mathcal{G}}\left(\left(\mathcal{L}(n_2)|_{\mathcal{L}(-n_1)W_2}\right)^{-1}\tilde{W}(n_1,n_2),Y_2\right)\\
			\leq - \min(|\lambda_{Y_2}-\lambda_{Y_1}|,|\lambda_{Y_2}-\lambda_{V_2}|).
		\end{multline}
		
	\end{lemma}

	\begin{proof}
		
		Let $W_1$ and $W_2$ be as in the claim. We apply \cref{lemma:SequenceEstimateFT2} to $(Y,V)=(Y_1,V_1)$ for $W=W_1$ and to $(Y,V)=(Y_1\oplus Y_2,V_2)$ for $W=W_2$ with their respective spaces and mappings at $-n$. It follows that
		\begin{equation*}
			\limsup_{n\to\infty}\frac{1}{n}\log\textnormal{d}_{\mathcal{G}}(\mathcal{L}(-n)W_1,Y_1)\leq -|\lambda_{Y_1}-\lambda_{V_1}|
		\end{equation*}
		and
		\begin{equation*}
			\limsup_{n\to\infty}\frac{1}{n}\log\textnormal{d}_{\mathcal{G}}(\mathcal{L}(-n)W_2,Y_1\oplus Y_2)\leq -|\lambda_{Y_2}-\lambda_{V_2}|.
		\end{equation*}
		Thus, we have good approximations of $Y_1$ and $Y_1\oplus Y_2$ from the past data. Moreover, by \cref{cor:SequenceEstimateFT} we have
		\begin{equation*}
			\limsup_{n\to\infty}\frac{1}{n}\log\|\Pi_{V_2||Y_1\oplus Y_2}|_{\mathcal{L}(-n)W_2}\|\leq -|\lambda_{Y_2}-\lambda_{V_2}|.
		\end{equation*}
		Since $\mathcal{L}(-n)W_1$ converges to $Y_1$, the projections $\Pi_{\mathcal{L}(-n)W_1||V_1}$ converge to $\Pi_{Y_1||V_1}$ by \cref{lemma:ProjectionStable}. In particular, $\|\Pi_{\mathcal{L}(-n)W_1||V_1}\|$ and $\|\Pi_{V_1||\mathcal{L}(-n)W_1}\|$ are bounded from above by a constant independent of $n$.\par
		
		The growth rate assumptions for future data imply
		\begin{equation*}
			\limsup_{n\to\infty}\frac{1}{n}\log\frac{\|\mathcal{L}(n)|_{V_1}\|}{\inf_{y\in Y_1\cap S}\|\mathcal{L}(n)y\|}\leq - |\lambda_{Y_1}-\lambda_{V_1}|.
		\end{equation*}
		Now, apply \cref{cor:MapEstimateBT} to $(Y_1,V_1)$, $(Y_2,V_2)$, the complements $\mathcal{L}(-n_1)W_1$ of $V_1$ and $\mathcal{L}(-n_1)W_2$ of $V_2$, $\mathcal{L}=\mathcal{L}(n_2)$, and $\tilde{W}=\left(\mathcal{L}(n_2)|_{\mathcal{L}(-n_1)W_2}\right)^{-1}\tilde{W}(n_1,n_2)$. We get
		\begin{multline}\label[equation]{eqn:SequenceEstimateBT3}
			\underset{\tilde{w}\in\left(\mathcal{L}(n_2)|_{\mathcal{L}(-n_1)W_2}\right)^{-1}\tilde{W}(n_1,n_2)\cap B}{\sup}\,\textnormal{d}(\tilde{w},Y_2\cap B)\leq\\
			2\left(\frac{2}{\delta}\|\Pi_{V_1||\mathcal{L}(-n_1)W_1}\|+\|\Pi_{V_1||Y_1}\|\,\|\Pi_{\mathcal{L}(-n_1)W_1||V_1}\|\right)\frac{\|\mathcal{L}(n_2)|_{V_1}\|}{\inf_{y\in Y_1\cap S}\|\mathcal{L}(n_2)y\|}\\
			+2\|\Pi_{V_2||Y_1\oplus Y_2}|_{\mathcal{L}(-n_1)W_2}\|.
		\end{multline}
		In view of \cref{lemma:SymmetryOfCloseness}, all that remains to prove \cref{eqn:SequenceEstimateBT2} is to insert respective asymptotics into the terms of \cref{eqn:SequenceEstimateBT3}. Indeed, the terms inside the large brackets are bounded from above by a constant, and the other terms can be estimated as above.
		
	\end{proof}

	\cref{lemma:SequenceEstimateBT} provides an appropriate tool to investigate convergence of the Ginelli algorithm. Since Ginelli's algorithm initiates vectors for the backward propagation inside spaces from the forward propagation, which vary with the chosen runtime, the domain for initial vectors is not constant. This poses a problem when talking about convergence with respect to initial conditions. One way to solve this problem is to express the initial vectors of the backward propagation in terms of runtime-independent coefficients. In other words, we want to find a family of isomorphisms identifying $\mathcal{L}(n_1)\mathcal{L}(-n_2)W_2$ with $\real^{k_1+k_2}$.\par 
	
	If $X=H$ is a Hilbert space, then we may identify an orthonormal basis of $\mathcal{L}(n_1)\mathcal{L}(-n_2)W_2$ with the standard basis of $(\real^{k_1+k_2}, \|.\|_2)$. The identification defines an isometry leaving distances and angles invariant. In particular, we may check \cref{eqn:SequenceEstimateBT} on the coefficient space.

\subsection{Convergence proof}\label{subsec:ConvergenceProof}
	
	In this subsection we combine our tools to prove \cref{thm:Convergence}.
	
	\begin{proof}[proof of \cref{thm:Convergence}]
		
		First, we set $\Omega'\subset\Omega$ to be the subset of full $\mathbb{P}$-measure on which the Oseledets splitting is defined and on which \cref{eqn:FiltrationUniformEstimateFT,eqn:FiltrationUniformEstimateFT2,eqn:SplittingUniformEstimateFT,eqn:FiltrationUniformEstimateBT,eqn:FiltrationUniformEstimateBT2,eqn:SplittingUniformEstimateBT} hold. Fix some $\omega\in\Omega'$.\par 
		
		Let $\mathbf{F}_j\subset H^{m_1+\dots+m_j}$ be the subset of all tuples inducing well-separating common complements of $(V_{j+1}(\sigma^{-n}\omega))_{n\in\natural}$ for $j=1,\dots,p$. Then, the set
		\begin{equation}
			\mathbf{F}:=\left(\mathbf{F}_1\times H^{m_2+\dots+m_p}\right)\cap\left(\mathbf{F}_2\times H^{m_3+\dots+m_p}\right)\cap\dots\cap \mathbf{F}_p\subset H^k
		\end{equation}
		consists of tuples $(x_{1_1},\dots,x_{p_{m_p}})$ such that $\textnormal{span}(x_{1_1},\dots,x_{j_{m_j}})$ is a well-separat{\-}ing common complement of $(V_{j+1}(\sigma^{-n}\omega))_{n\in\natural}$ for each $j=1,\dots,p$. In particular, since products and intersections of prevalent sets are prevalent, \cref{thm:WellSeparating} implies that $\mathbf{F}$ is prevalent. We use elements of $\mathbf{F}$ as initial vectors for the forward propagation in Ginelli's algorithm.\par 
		
		Let $\mathbf{B}\subset\real^{k\times k}_{ru}$ be the subset of upper diagonal matrices with non-zero diagonal elements, i.e., the subset of invertible upper diagonal matrices. $\mathbf{B}$ has full Lebesgue measure and is used for initial vectors for the backward propagation in Ginelli's algorithm.\par 
		
		Now, let $((x_1,\dots,x_k),R)\in\mathbf{F}\times\mathbf{B}$ be an input for Ginelli's algorithm. According to \cref{thm:ConvergenceBSFT} the first set of vectors $(x_{1_1},\dots,x_{1_{m_1}})$ gives an approximation of $Y_1(\omega)$ via the first step of Ginelli's algorithm. The remaining steps of Ginelli's algorithm do not change this approximation. In fact, the first set of output vectors $(G_{\omega,k}^{n_1,n_2})_{1_i}$ for $i=1,\dots,m_1$ at $((x_1,\dots,x_k),R)$ spans the same space as $\left(\mathcal{L}_{\sigma^{-n_1}\omega}^{(n_1)}x_1,\dots,\mathcal{L}_{\sigma^{-n_1}\omega}^{(n_1)}x_k\right)$. Thus, we have
		\small
		\begin{multline*}
			\limsup_{N\to\infty}\underset{n_1,n_2\geq 	N}{\sup}\,\frac{1}{\min(n_1,n_2)}\log\textnormal{d}_{\mathcal{G}}\left(\textnormal{span}\left\{\left.\left(G_{\omega,k}^{n_1,n_2}\right)_{1_i}\ \right|\ i=1,\dots,m_1\right\},Y_1(\omega)\right)\\
			\leq -|\lambda_1-\lambda_2|= -\min\left(|\lambda_1-\lambda_0|,|\lambda_1-\lambda_2|\right)
		\end{multline*} 	
		\normalsize
		at $((x_1,\dots,x_k),R)$.\par 
		
		Convergence of the remaining spaces is due to \cref{lemma:SequenceEstimateBT}. Indeed, fix some $1<j\leq p$. We set $Y_1=Y_1(\omega)\oplus\dots\oplus Y_{j-1}(\omega)$, $V_1=V_j(\omega)$, $Y_2=Y_j(\omega)$, $V_2=V_{j+1}(\omega)$, $\mathcal{L}(-n)=\mathcal{L}_{\sigma^{-n}\omega}^{(n)}$, $\mathcal{L}(n)=\mathcal{L}_{\omega}^{(n)}$, and spaces $Y_i(\pm n)$ and $V_i(\pm n)$ for $i=1,2$ accordingly. The growth rates in \cref{lemma:SequenceEstimateBT} are given by \cref{thm:MET} and its proof. Furthermore, let $W_1=\textnormal{span}(x_{1_1},\dots,x_{(j-1)_{m_{j-1}}})$ and $W_2=\textnormal{span}(x_{1_1},\dots,x_{j_{m_j}})$ be the well-separating common complements, which approximate $Y_1$ and $Y_2$ in the first step of Ginelli's algorithm. The family of spaces $(\tilde{W}(n_1,n_2))_{n_1,n_2\in\natural}$ is given by $\textnormal{span}(y_{j_1}^1,\dots,y_{j_{m_j}}^1)$ via vectors of the fourth step of the algorithm. Indeed, the $j_1^{\textnormal{th}}$ to $j_{m_j}^{\textnormal{th}}$ column 
		\begin{equation*}
			[r_{j_1}|\dots|r_{j_{m_j}}]:=
			\begin{bmatrix}
				
					*           & \cdots & *                   \\
					\vdots      &        & \vdots              \\
					*           & \cdots & *                   \\
					r_{j_1,j_1} & \cdots & r_{j_1,j_{m_j}}     \\
					            & \ddots & \vdots              \\
					0           &        & r_{j_{m_j},j_{m_j}} \\
					0           & \cdots & 0                   \\
					\vdots      &        & \vdots              \\
					0           & \cdots & 0 
					
			\end{bmatrix}
		\end{equation*}
		of $R$ give us coefficients with which we may express $y_{j_1}^1,\dots,y_{j_{m_j}}^1$ in terms of the orthonormalized vectors
		\begin{multline*}
			GS(\mathcal{L}(n_2)\mathcal{L}(-n_1)x_{1_1},\dots,\mathcal{L}(n_2)\mathcal{L}(-n_1)x_{j_{m_j}})\\
			=GS\left(\mathcal{L}_{\sigma^{-n_1}\omega}^{(n_1+n_2)}x_{1_1},\dots,\mathcal{L}_{\sigma^{-n_1}\omega}^{(n_1+n_2)}x_{j_{m_j}}\right),
		\end{multline*}
		which emerge in the third step of Ginelli's algorithm. Through the identification via $GS$, \cref{eqn:SequenceEstimateBT} may be checked on coefficient space. Since $\mathcal{L}(n_2)\mathcal{L}(-n_1)W_1$ is mapped to $\real^{m_1+\dots+m_{j-1}}\times\{0\}\subset\real^k$ and $\mathcal{L}(n_2)\mathcal{L}(-n_1)W_2$ to $\real^{m_1+\dots+m_{j}}\times\{0\}\subset\real^k$, we need to check if
		\begin{equation*}
			\underset{r\in\textnormal{span}\left(r_{j_1},\dots,r_{j_{m_j}}\right)\cap S}{\inf}\,\|\Pi_jr\|>0,
		\end{equation*}
		where $\Pi_j:\real^k\to\{0\}\times\real^{m_j}\times\{0\}$ is the projection onto the $j_1^{\textnormal{th}}$ to $j_{m_j}^{\textnormal{th}}$ coordinates. This is easily verified, since $R$ is an upper diagonal matrix with non-zero elements on the diagonal. Thus, we may apply \cref{lemma:SequenceEstimateBT} to see that the span of the $j_1^{\textnormal{th}}$ to $j_{m_j}^{\textnormal{th}}$ vector from the fifth step of Ginelli's algorithm approximates $Y_j(\omega)$ at the desired speed. This concludes the proof.\footnote{The last step of Ginelli's algorithm only normalizes computed vectors. It does not change their linear span and, thus, plays no role in \cref{eqn:ConvergenceProof}. However, the step is a necessary part of the algorithm, since CLVs are defined as normalized basis vectors of $Y_j(\omega)$.}
		
	\end{proof}

\bigskip
	
	\section{Conclusions}\label{sec:Conclusions}

	With the emergence of semi-invertible METs, the concept of CLVs has been opened up to new settings. In particular, various infinite-dimensional versions of the MET have been proved. In this article we followed the semi-invertible MET from \cite{GonzalezTokmanQuas.2014} to generalize Ginelli's algorithm for CLVs. Our main result is a convergence proof of the algorithm in the context of Hilbert spaces. The proof not only generalizes previous analysis of Ginelli's algorithm and its features \cite{Noethen.2019,GinelliChateLiviPoliti.2013,ErshovPotapov.1998} to an infinite-dimensional setting, but also treats the case of non-invertible linear propagators. We formulated most arguments in the context of maps on Banach spaces before connecting them to basic asymptotic properties of the Oseledets splitting. Since those properties appear in most versions of the MET, our convergence proof may be translated to other settings as well.\par 
	
	We split the proof into estimates for forward and for backward propagation. It turned out that, during forward propagation, almost every complement of spaces of the Oseledets filtration asymptotically aligns with the Oseledets spaces. The fact that complements generically align in forward-time even holds if we only have an Oseledets filtration. For backward propagation, we had to restrict the propagator to certain subspaces, since it may not be globally invertible in a semi-invertible setting. Last but not least, we combined our estimates to form the convergence proof.\par 
	
	Throughout the proof, we connected estimates to the Lyapunov exponents that appear in the MET. Thus, we were able to relate Lyapunov exponents to the speed of convergence. As for the finite-dimensional case in \cite{Noethen.2019}, Ginelli's algorithm converges exponentially fast with a rate given by the spectral gap between associated Lyapunov exponents. However, note that the notation of our convergence theorem excludes subexponential prefactors of the speed of convergence. Especially in view of applications, those prefactors may very well be important. However, they depend on the particular system, and their derivation requires an in-depth analysis of the proof of the MET. Since we aimed for a dynamical approach that is not tailored to only one version of the MET, we leave the analysis of subexponential prefactors to future research.\par 
	
	While we successfully generalized and proved Ginelli's algorithm for infinite dimensions, it is primarily an analytical tool. The numerical computation of CLVs brings its own set of challenges. Indeed, our results may be seen as a help to understand limit cases of applications of Ginelli's algorithm for systems of increasingly higher resolutions. The transition between finite and infinite dimensions is still an open question and leads to the concept of stability of CLVs. Additionally, numerical inaccuracies in computing the linear propagator may result in a different output of Ginelli's algorithm. In fact, the MET only guarantees that CLVs depend measurably on the trajectory.\par 
	
	Despite the remaining challenges, we made a big step towards computing CLVs in infinite dimensions. Through the connection to semi-invertible METs, our research applies to recent developments in the context of CLVs and paves the way for new advancements of both analytical and numerical aspects of CLV-algorithms.

\bigskip
	
	\section*{Acknowledgments}

	This paper is a contribution to the project M1 (Instabilities across scales and statistical mechanics of multi-scale GFD systems) of the Collaborative Research Centre TRR 181 "Energy Transfer in Atmosphere and Ocean" funded by the Deutsche Forschungsgemeinschaft (DFG, German Research Foundation) - Projektnummer 274762653.

\bigskip
	
	\printbibliography

	
\end{document}